\documentclass{article}
\usepackage{amssymb}
\usepackage{multicol}
\usepackage{graphicx}
\usepackage{amsthm}
\usepackage{verbatim}
\usepackage{color}

\usepackage{amsmath, epsfig}
\newtheorem{theorem}{Theorem}
\newtheorem{lemma}{Lemma}
\newtheorem{proposition}{Proposition}

\theoremstyle{definition}
\newtheorem{definition}{Definition}
\newtheorem{remark}{Remark}

\newcommand{\FF}{\mathcal{F}}
\newcommand{\expo}{{\rm expo}}
\newcommand{\TT}{\mathcal{T}}
\newcommand{\FI}{\mathcal{FI}}
\newcommand{\II}{\mathcal{I}}
\newcommand{\Tti}{\tilde{T}}

\newcommand{\p}{\mathbb{P}}

\renewcommand{\phi}{\varphi}

\newcommand{\ignore}[1]{}
\title{Two Layer 3D Floor Planning}
\author{Paul Horn\thanks{Department of Mathematics, Harvard University, {\tt $\{$phorn, lippner$\}$@math.harvard.edu}\newline Research supported by grant FA9550-09-1-0090-DOD-35-CAP} \and Gabor Lippner$^*$}

\begin{document}
\maketitle
\begin{abstract}
A 3D floor plan is a non-overlapping arrangement of blocks within a large box.  Floor planning is a central notion in chip-design, and with recent advances in 3D integrated circuits, understanding 3D floor plans has become important.  In this paper, we study so called mosaic 3D floor plans where the interior blocks partition the host box under a topological equivalence.  We give representations which give an upper bound on the number of general 3D floor plans, and further consider the number of two layer mosaic floorplans.  We prove that the number of two layer mosaic floor plans is $n^{(1+o(1))n/3}$.  This contrasts with previous work which has studied `corner free' mosaic floor plans, where the number is just exponential.  The upper bound is by giving a representation, while the lower bound is a randomized construction.
\end{abstract}

\section{Introduction}

A (2D) floor plan is non-overlapping arrangement of rectangles within a larger host rectangle. It is a central notion in chip-design, where one wants to find an (in some sense) optimal arrangement of the components of a chip. The state-of-the-art methods for finding these optimal arrangements often involve enumerating through all possible arrangements, or generating arrangements randomly to find the optimal one. Hence finding efficient encodings for arrangements with a given number, $n$, of rectangles - and understanding the limitations of these - has become an important issue in the field of chip-design. Recently technology has enabled us to build chips in three dimensions~\cite{3dC}. This calls for an understanding of the 3D version of floor plans, where blocks are arranged within a host box.

Mosaics are a special type of floor plans, where the smaller rectangles partition the host rectangle. Generic mosaic floor plans are such that they do not have adjacencies that could be removed by a small perturbation (the precise definition will be given in Section~\ref{floorplans_section}).  Enumeration and efficient encoding of these is the first step towards understanding general floor plans. Obviously, there are uncountably many different (mosaic) floor plans. To meaningfully approach these questions one has to start with a notion of equivalence of floor plans, and enumerate and encode the resulting equivalence classes. We will consider the most widely studied notion, the so called topological equivalence (see Definition~\ref{equivalence_defs}). The main problems are to find efficient encodings and asymptotic enumeration of topological equivalence classes of mosaic floor plans.

In the 2D case these problems have been studied extensively and satisfying answers have been found. The number of equivalence classes is essentially exponential in $n$. In fact they can be completely enumerated and their number turns out to be the number of Baxter permutations. There are nice bijections between floor plans and such permutations (see e.g.~\cite{baxter}),  and also between floor plans and pairs of dual binary trees (see~\cite{dual trees}) that provide efficient encodings. The crucial observation behind these results is always a certain kind of induction on the number of rectangles: if one removes the top left rectangle from an arrangement, the resulting gap can be filled by extending some of the other rectangles in the arrangement that were previously adjacent to the rectangle just removed.

This idea does not carry over to 3D floor plans. There is a certain kind of local arrangement (see Figure~\ref{fig:corner}) where removing the small block in the front, neither of the adjacent blocks can be extended to cover the gap without creating overlaps. We shall refer to such an arrangement as a \textit{diagonal corner}. This seems to be the only obstruction to generalizing the inductive argument: if one forbids diagonal corners, then a very similar induction works,  see e.g. Chapter 3.3 of~\cite{thesis?}. Hence many results carry over to 3D, in particular the number of such arrangements is again (up to lower order multiplicative terms) exponential. However the question is completely unresolved for general 3D arrangements. It was previously unknown even whether there exist more than exponentially many 3D floor plans.  Indeed we are unaware of any non-trivial lower or upper bounds on the number of 3D mosaic floor plans.

\begin{figure}[ht]
\begin{center}
\includegraphics[height=2in]{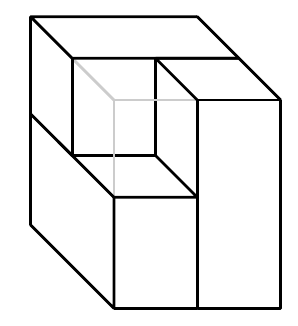}
\end{center}
\caption{Diagonal Corner\label{fig:corner}}
\end{figure}

A number of other representations have been proposed for (often special classes of) 3D floor plans.  A recent survey is given by \cite{flk}.  This lists three attempts to represent mosaic floorplans.  The 3D-corner block lists of Ma, Hong, Dong and Cheng~\cite{3dcorn} and O-Sequence of Ohta, Yamada, Kodama, and Fujiyosa~\cite{OSeq} give representations for labeled floorplans without diagonal corners.   Finally, there has been an attempt by Wang, Young, and Cheng~\cite{Cheng1} to generalize pairs of trees to encode 3D mosaic floor plans. Their approach would give an $n^n$ type upper bound, but their paper has errors and their encoding appears to be incomplete.

The starting point of our paper is nevertheless~\cite{Cheng1}.
Let $\FF_n$ denote the set of unlabeled generic 3D mosaic floor plans, and $\TT_n$ the subset where there are only 2 layers in the $z$ direction. Our main contributions are the following.
\begin{itemize}
\item We give a non-trivial upper bound on the number of generic 3D floor plans.
\begin{theorem}\label{upperbound_theorem}
\[
\log |\FF_n| \leq 3n \log n + O(n).
\]
\end{theorem}

\item We show that already the number of generic two-layer 3D mosaic floor plans up to topological equivalence is super-exponential. In particular using a random construction we prove
\begin{theorem}\label{randomconstr_theorem}
\[
\log |\TT_n| \geq \left(\frac{1}{3}-o(1)\right)n \log n.
\]
\end{theorem}

\item Finally,  exploiting duality of binary trees for 2D floor plans, careful analysis gives that this lower bound is asymptotically correct for 2-layer floor plans.
\begin{theorem}\label{2layer_theorem}
\[
\lim \frac{\log |\TT_n|}{n \log n} = \frac{1}{3}.
\]
\end{theorem}

\end{itemize}

\begin{remark} We find it rather surprising that already 2-layer floor plans have enough complexity to yield a super-exponential number of equivalence classes, especially in the light of the result that the number of general 3D plans without diagonal corners is only exponential. On the other hand we haven't been able to exploit multiple layers to construct significantly more floor plans. The best lower bound we have for $\log |\FF_n|$ is, asymptotically,  $\log |\TT_n|$.
\end{remark}

The paper is organized as follows. In Section~\ref{floorplans_section} we give precise definitions of the various notions related to floor plans. We also recall the quaternary corner-tree construction of Wang et. al. from~\cite{Cheng1}. In Section~\ref{lowerbound_section} we give a random construction to prove Theorem~\ref{randomconstr_theorem}. In Section~\ref{representation_section} we show how 8-tuples of labeled corner-trees can be used to encode a 3D floor plan. Then an easy observation is used to reduce the required number of trees from 8 to 4 and finish the proof of Theorem~\ref{upperbound_theorem}. Finally in Section~\ref{reconstruction_section} we prove Theorem~\ref{2layer_theorem} by showing that in the case of 2-layer floor plans much less information is sufficient to encode the labelings on the trees.

\section{Floor plans}\label{floorplans_section}

A mosaic floor plan is a subdivision of a large host box into $n$ smaller blocks. We shall only be interested in certain special kinds of floor plans called generic floor plans. A point in the host box will be referred to as a \textit{corner} if it is the vertex of at least one small block.  Figure~\ref{fig:corners} shows examples of local configurations of blocks adjacent a corner.

\begin{definition}
A corner is \textit{generic} if its "shape" is stable under any small perturbation of the blocks adjacent to it.  (We do not require the perturbation to be extendable to the whole arrangement, so this is a local notion!)
A floor plan is \textit{generic} if all its corners are.
\end{definition}

\begin{figure}[ht]
\begin{center}
\includegraphics[height=2in]{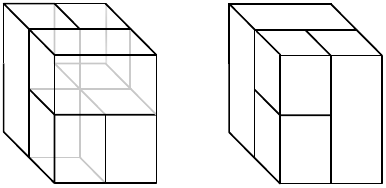}
\end{center}
\caption{Generic corners (in the middle of the cubes) \label{fig:corners}}
\end{figure}

\begin{remark} It is not hard to see that a non-generic 2D floor plan can always be perturbed into a generic one. The same is not true however for 3D floor plans. Figure~\ref{fig:nongeneric} shows a non-generic corner along with a local perturbation that changes it. Figure~\ref{fig:nongeneric2} shows a 3D floor plan which has such non-generic corners but that is globally rigid. Thus this notion of genericity is, in some sense, "incorrect". The main reason we still use this definition is because the properties of generic floor plans are much nicer compared to the non-generic ones.
\end{remark}

\begin{figure}[ht]
\centering
\begin{center}\def\svgwidth{2in}
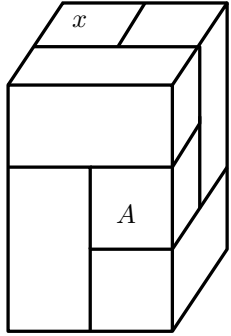
\end{center}
\caption{A globally rigid but configuration where the interior corners of the block A are non-generic. (The block labeled x reaches all the way to the bottom.)  
\label{fig:nongeneric2}}
\end{figure}

\subsection{Dissection planes}

There is a classical encoding of floor plans which we now recall. One can reconstruct a floor plan given the appropriate coordinates of each face of each box. (The $x$-coordinate of the sides orthogonal to the $x$-axis, etc.)  These coordinates cannot be arbitrary, however. When two blocks have touching faces, these sides obviously have to have the same coordinate value. Hence it is enough to store one value for each group of faces that must have the same coordinate. This motivates the following approach.

\begin{definition}\label{dissection_plane}~
\begin{itemize}
\item A \textit{dissection plane} (of a given floor plan) is the union of a maximal set of faces that must have the same coordinate value. More precisely let us say that two faces are attached if their intersection has positive area. Taking the equivalence relation generated by this, dissection planes are exactly the unions of equivalence classes.
\item The \textit{dissection graph} associated to a floor plan in a coordinate direction is an oriented graph. Its vertices are the dissection planes orthogonal to the given direction. For each block of the floor plan one draws an edge between the dissection planes containing the two faces of the block orthogonal to the given direction, and orients it towards the plane with the larger coordinate value.
\end{itemize}
\end{definition}

\begin{definition}\label{equivalence_defs}
Two floor plans are \textit{topologically equivalent} if their dissection graphs are isomorphic.
\end{definition}

\begin{figure}[ht]
\begin{center}
\includegraphics[height=2in]{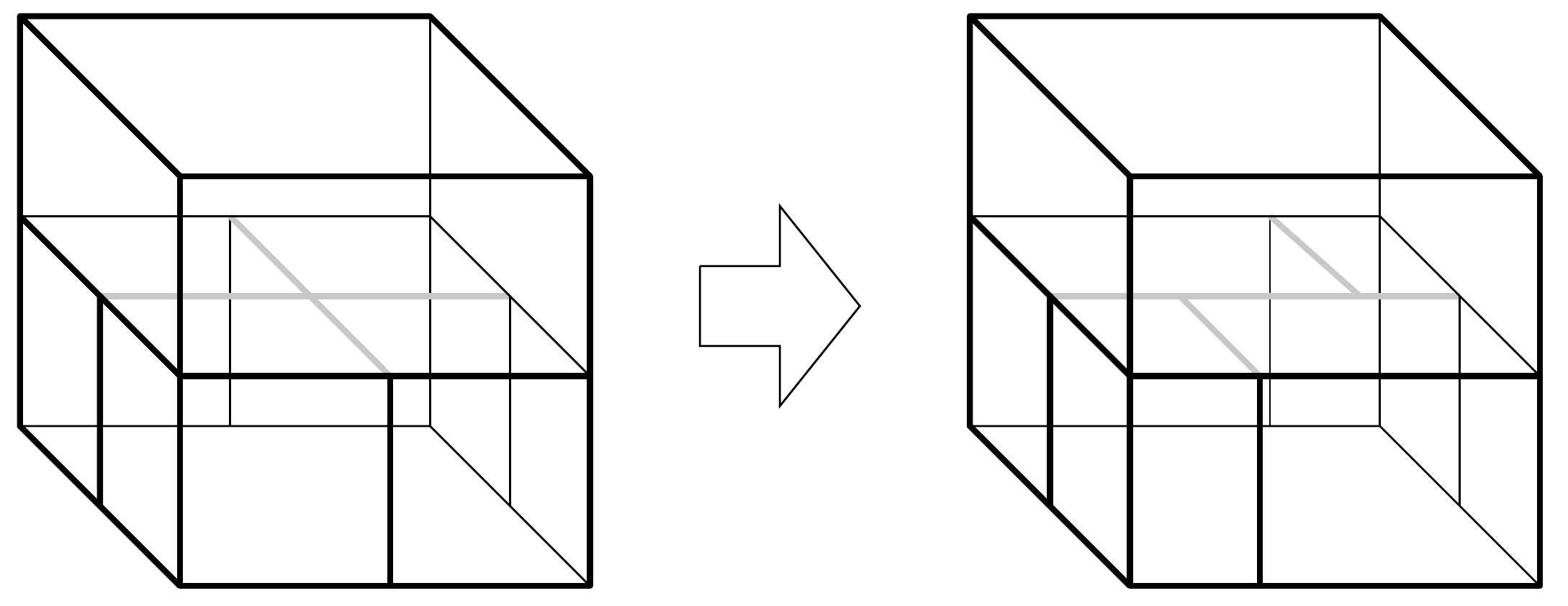}
\end{center}
\caption{A non-generic corner and its perturbation \label{fig:nongeneric}}
\end{figure}

\begin{remark}
The name 'topological equivalence' comes from the following observation. Imagine that the walls of the boxes are mobile, as in a Japanese house. A deformation of a floor plan consists of a sequence of translations of walls, without ever creating empty spaces, and such that each moves starts and ends in a locally generic configuration.  It is easy to see that deformations preserve the dissection graphs, but it is a well known fact that the converse also holds: floor plans with isomorphic dissection graphs can be deformed into each other.

 On the other hand, given the dissection graphs, one can reconstruct the equivalence class of the floor plan as follows. The dissection graphs determine a partial order on the set of dissection planes. Extending this arbitrarily to a complete order, and assigning coordinate values respecting this order determines a choice for the coordinates of each block, hence one gets the desired floor plan.
\end{remark}

\begin{definition}\label{def:segment} We will also use a one-dimensional analogue of dissection planes. Similarly to faces, we can say that two edges (of two blocks) are attached if their intersection has positive length. Extending this to an equivalence relation, unions of the equivalence classes are going to be referred to as \textit{supporting segments}. It is not hard to see (by looking at Figure~\ref{fig:corners}) that, when two edges are attached, they lie in the intersection of two perpendicular dissection planes.
\end{definition}

\subsection{Corner trees}

We shall need one more construction from~\cite{Cheng1} that will be used to encode floor plans, the so-called corner tree. Let us choose a corner of the host box and denote it by $P$. Each block $A$ in the floor plan has a corresponding corner, the one closest to $P$.  Examining the possible shapes of a generic corner, it is not hard to see that there is always exactly one other block $B$ that has the same corner, unless this corner is $P$ itself. There are 4 possibilities for relative position of the second block with respect to the first block. They are either touching along a face in either of the three directions, or they intersect only in the corner itself.  We say that, with respect to $P$, the block $B$ is the ($x$/$y$/$z$/$d$)-parent of $A$ in the four cases respectively (where $d$ stands for diagonal). Figure~\ref{fig2} shows the four types of adjacencies.

\begin{definition}\label{corner_tree}
Suppose we are given a floor plan with the blocks labeled from $1$ to $n$. For each corner $P$ of the host box, let $T_P$ be a tree whose vertices are the blocks, and each block is connected by an edge to its unique parent with respect to $P$. Vertices are labeled $1$ to $n$ according to the numbers in the floor plan, and edges are labeled according to the type of the parent relation. The edge and vertex labeled tree constructed this way is called a \textit{corner tree} of the floor plan. The tree has a unique root, the block that is adjacent to the corner $P$. It will also be useful to consider the corner tree without the vertex labels. This will be denoted by $\Tti_P$.
\end{definition}

It is easy to see that corner trees are invariant under deformations, hence topologically equivalent floor plans have isomorphic corner trees. Furthermore, every vertex in a corner tree can have at most four children, one for each type of edge. Thus a corner tree is a rooted edge-labeled quaternary tree on the set $\{1,2,\dots, n\}$ with each vertex having at most one child for each of the four possible labels.

The labeling of the blocks not does not make a huge difference in this construction. For each floor plan there are exactly $n!$ ways to add the labels, hence enumerating labeled or unlabeled floor plans is equivalent. To encode labeled floor plans, we will use labeled corner trees. To encode unlabeled floor plans one can use the same set of trees, but with the labeling of the first tree arbitrary, and the labels are used only as an identification of vertices among trees. In other words, instead of considering eight trees with vertex labels, one can think of the eight trees given on the same vertex set, without any labels. This approach will be particularly useful when dealing with 2-layer floor plans.

In 2D it turns out that two corner trees corresponding to two opposite corners completely determine the dissection graphs of the floor plan, hence also its topological equivalence class. Furthermore, rather surprisingly, the vertex labeling of one tree can be reconstructed from the vertex labeling of the other tree, hence only one of the trees have to be vertex labeled. Since the blocks of a floor plan can be numbered arbitrarily, this means that 2D floor plans can be encoded with a pair of binary trees with only edge-labels.  Since the number of (pairs of) binary trees is at most exponential in $n$, so is the number of floor plans.  This is, however, not the case in 3D.

In~\cite{Cheng1} the authors claim that, analogously to the 2D case, a pair of (vertex labeled) corner trees corresponding to two opposite corners completely determine the dissection graphs of the floor plan.   Since there are exponentially many pairs of quaternary trees, and $n!$ identifications between vertices in the pair this would give an upper bound of the form
$n! \cdot \expo(n)= n^{(1+o(1))n}$. However their proof is incomplete and we do not presently see if the gap can be filled.

\section{Lower bound construction}\label{lowerbound_section}

In this section we prove Theorem~\ref{randomconstr_theorem} by constructing so many distinct floor plans.

Consider an $N \times N$ grid, and $S \subseteq [N]\times [N]$.  Construct a floorplan $F_S$ as follows:

\begin{figure}[ht]
\begin{center}
\includegraphics[height=2in]{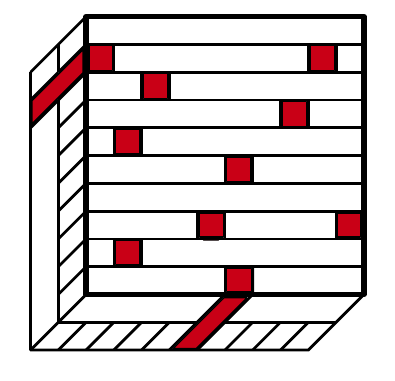}
\end{center}
\caption{Vertical blocks in the positions of $S$}
\end{figure}

Place a height two block at each position in $S$.  The remainder of the top level is filled in with horizontal strips, while the bottom level is
filled in with vertical strips.  For each set $S$, this yields a unique (though not necessarily generic) floorplan.  Let us now count the total
number of blocks.  This floor plan has $|S|$ blocks of height $2$.  If $|S|=0$, there are $N$ vertical and $N$ horizontal strips.  Meanwhile, each element in $S$
which is not on the boundary of the grid, and not immediately adjacent to another element in $S$ adds an additional vertical and horizontal strip.  We
say that $S$ is {\it good} if this holds for every element in $S$.  If $S$ is good, then there are $N + |S|$ horizonal and vertical strips, thus in total
there are $2N + 3|S|$ blocks.  If, in addition, there are no elements of $S$ that are diagonally incident then the floorplan $F_S$ is generic, thus
$F_S$ represents an equivalence class in $\TT_{2N + 3|S|}$.

It is clear that $S$ can be recovered from $F_S$.  Thus the number of good $S$ where there are no diagonally incident elements gives a lower bound on $|\TT_n|$.

\begin{lemma} \label{prop1}
Fix $0 \leq t \leq \frac{(N-2)^2}{10}$ and $N$ sufficiently large.  Let $S \subset \{2, \dots, N-1\} \times \{2, \dots, N-1\}$ be a uniformly randomly chosen set of size $t$.  Then
\[
\p(\mbox{$S$ good and $F_S$ generic}) \geq \exp\left( - 50 \frac{t^2}{N^2}\right).
\]
\end{lemma}
\begin{proof}
Choose $S$ to be a uniformly randomly chosen {\it ordered} set of size $t$.  Let $S_i$ denote the first $i$ elements of $S$, and let $\mathcal{A}_i$ denote the event that $S_i$ is good, and $F_{S_i}$ is generic.
\begin{align}
\p(\mathcal{A}_t) &= \p(\mathcal{A}_1)\p(\mathcal{A}_2|\mathcal{A}_1)\p(\mathcal{A}_3|\mathcal{A}_2) \cdots \p(\mathcal{A}_t|\mathcal{A}_{t-1}) \nonumber \\
& \geq 1 \cdot \left(1 - \frac{9}{(N-2)^2}\right) \cdot \left(1 - \frac{18}{(N-2)^2}\right) \cdots \left(1 - \frac{9(t-1)}{(N-2)^2}\right) \nonumber \\
& \geq \exp\left( - \sum_{i=1}^{t-1} \frac{9i}{(N-2)^2-9i} \right) \nonumber \\
& \geq \exp\left( - \sum_{i=1}^{t-1} \frac{9i}{(N-2)^2-9t} \right) = \exp\left( - \frac{9t(t-1)}{2((N-2)^2 -9t)} \right). \label{eqn1}
\end{align}
Here, the first inequality comes from the fact that the first $i$ blocks forbid at most $9i$ blocks for the $i+1$st choice, and the second inequality comes from the real number inequality
$(1-\frac{1}{x}) \geq e^{-\frac{1}{x-1}}$ for $x > 1$.  The result then follows noting that we used the fact that $t \leq \frac{(N-2)^2}{10}$ and the fact that $N$ is sufficiently large to somewhat simplify (\ref{eqn1}).
\end{proof}

To complete the construction note that for a fixed $N$ and $t$, where $t < \frac{(N-2)^2}{10}$ and $n = 2N+3t$, that

\begin{align}
|\TT_n| \geq \exp\left(-50\frac{t^2}{N^2}\right) {(N-2)^2 \choose t} &\geq \exp\left(-50 \frac{t^2}{N^2} \right) \left( \frac{(N-2)^2}{t}\right)^t \nonumber\\
&= \exp\left( t \log \left( \frac{(N-2)^2}{t}\right) - 50\frac{t^2}{N^2} \right). \label{eqn2}
\end{align}

Fix $t = N \log N$ and note that in this regime $n = (3+o(1))t$ and $\log(n) = (1+o(1)) \log(\frac{N}{\log N})$.  Using this in (\ref{eqn2}) gives the desired bound:
\begin{align*}
|\log \TT_n| \geq \left(\frac{1}{3}-o(1)\right) n \log n.
\end{align*}

\section{Tree representation scheme}\label{representation_section}

In this section we prove Theorem~\ref{upperbound_theorem} by giving a representation scheme for floor plans in $\FF_n$. We do this in two steps. First we show that the dissection graphs of a floor plan can be reconstructed from the eight-tuple of corner trees $(T_1, \dots, T_8)$ associated to the eight corners of the host box (see Figure~\ref{fig1} for the numbering of the corners). Then we show that knowing only $T_1, T_3, T_6$, and $T_8$, we can recover the four missing trees. The number of edge-labeled quaternary trees on $n$ vertices is exponential in $n$, hence this way we encode numbered floor plans with objects in a set of size at most $(n!)^4 c^n$. But numbered floor plans are exactly $n!$ as many as plain floor plans. Thus we get  $|\FF_n| \leq n!^3 c^n$ or equivalently
\[ \log |\FF_n| \leq 3n\log n + O(n).\]

\begin{figure}[ht]
\centering
\begin{center}\def\svgwidth{2in}
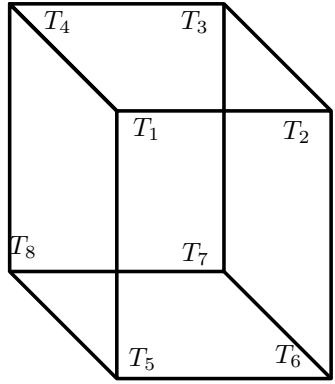
\end{center}
\caption{Corner/Tree identification\label{fig1}}
\end{figure}

\begin{figure}[ht]
\centering
\begin{center}
$\begin{array}{cccc}
\includegraphics[width=1.5in]{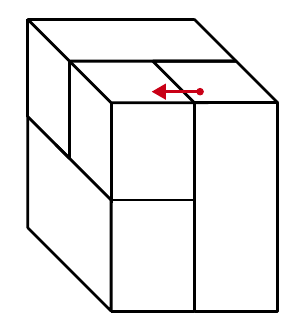} &
\includegraphics[width=1.5in]{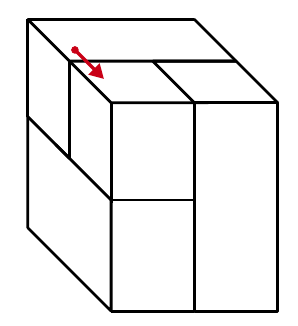} &
\includegraphics[width=1.5in]{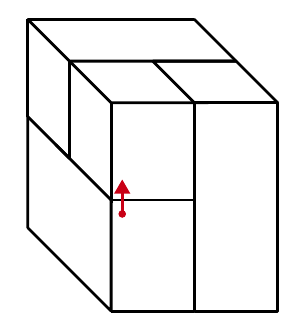} &
\includegraphics[width=1.5in]{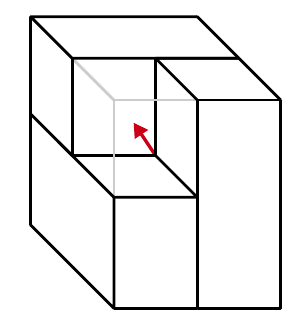}\\
\mbox{$x$-parent} & \mbox{$y$-parent} & \mbox{$z$-parent} & \mbox{$d$-parent}
\end{array}
$
\caption{Parent edge labels in $T_1$ tree\label{fig2}}
\end{center}
\end{figure}

\begin{proposition} The dissection graphs can be reconstructed from the corner trees $T_1, \dots, T_8$.
\end{proposition}

\begin{proof}
The vertices of the dissection graphs are the dissection planes, and edges are given by the blocks. If we know the list of dissection planes, and which face of which block belongs to which dissection plane, then we know the graphs themselves. Hence it is enough to figure out which face is in which dissection plane. But since dissection planes are defined as unions of faces in an equivalence class, it is enough to recover these equivalence classes.  Even though the equivalence classes were defined via the attachment relation, we cannot directly recover attachment from the list of corner trees. Instead, we shall use supporting segments (Definition~\ref{def:segment}) as an auxiliary tool.

First we shall show that, from the list of corner trees, one can recover which edges of which faces belong to the same supporting segment. This can be done by tracing a cycle of blocks around a supporting segment. Suppose we want to trace the segment that contains edge $X_0 X_1$  of block $B_0$. (See Figure~\ref{fig:segment}.) Then we look up $B_0$ in the tree corresponding to the $X_1$ corner, and find its parent $B_1$. Depending on what type of parent $B_1$ is, one of its edges, $X_1 X_2$,  will belong to the same supporting segment as $X_0 X_1$. Then we take the tree corresponding to the $X_2$ corner of $B_1$, find the parent of $B_1$, and denote it by $B_2$. The edge $X_2 X_3$ in $B_2$ then has to belong to the same supporting segment. When we reach the end of the segment, the process naturally "turns around", and starts finding blocks backwards along the segment. Then eventually it hits the other end of the segment and starts coming back, until it finally hits $B_0$ again. At this point we have a complete list of blocks adjacent to the supporting segment. Also, for each block, we know which edge of it belongs to the segment. This whole process is depicted in Figure~\ref{fig:segment}.

\begin{figure}[ht]
\begin{center}
\def\svgwidth{5in}
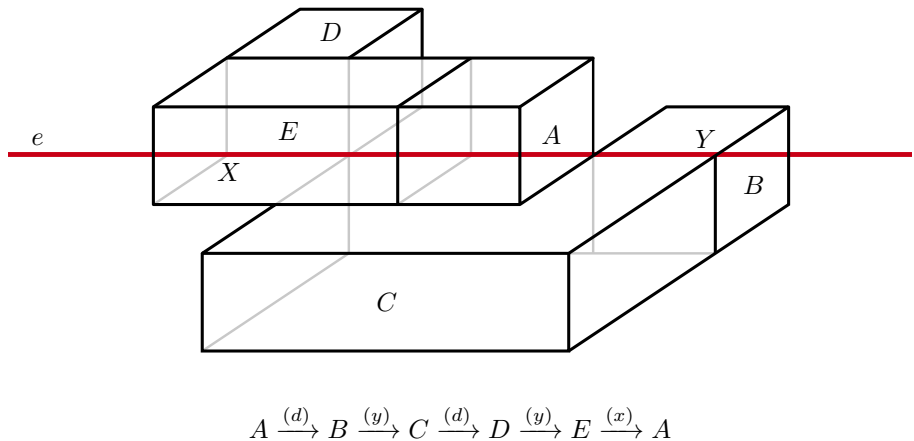
 \[A \xrightarrow{(d)} B \xrightarrow{(y)} C \xrightarrow{(d)} D \xrightarrow{(y)} E \xrightarrow{(x)} A \]
\end{center}
\caption{Tracing the supporting segment ${e=XY}$ \label{fig:segment}}
\end{figure}

Now we are ready to recover dissection planes.  Each dissection plane has two sides. By the definition of attachment, a dissection plane has connected interior. Thus, any face adjacent to one side can be reached from any other face adjacent to the same side by jumping from face to adjacent face. (Two faces are adjacent if they share an edge.) But when two faces are adjacent, their intersecting edges necessarily belong to the same supporting segment. Hence, starting from a face adjacent to a supporting segment, using the tracing method described above, one can find all other parallel faces that are adjacent to the same supporting segment. (And all such faces are necessarily contained in the same dissection plane!) Iterating this method, by connectivity, one eventually finds all faces on one side of a dissection plane. Finally the sides of a dissection plane can be connected together by tracing around a supporting segment on the boundary.

\end{proof}

\begin{proposition}[4 trees suffice]\label{4trees}  $T_2, T_4,T_5$ and $T_7$ are determined by $T_1,T_3,T_6$ and $T_8$.
\end{proposition}
\begin{proof}
The main observation is that there is a bijection between edges of the four unknown trees and the edges of the four given trees as follows. If there is a diagonal edge in $T_1$, say block $A$ is the diagonal parent of block $B$, then clearly there is a diagonal edge in $T_7$, namely block $B$ is the diagonal parent of block $A$ in that tree. And vice versa. Similarly, by looking at figure~\ref{fig2} one can see that if and only if $A$ is an ($x$/$y$/$z$)-parent of $B$ in $T_1$, if and only if is then $B$ an ($x$/$y$/$z$)-parent of $A$ in $T_2$/$T_4$/$T_5$ respectively.  Hence each edge in a tree can be recovered from the trees corresponding to the three adjacent vertices and the one opposite vertex. This proves the proposition, and also completes the proof of Theorem~\ref{upperbound_theorem}.
\end{proof}

\section{Label reconstruction}\label{reconstruction_section}

In this section we prove Theorem~\ref{2layer_theorem}. Given a two layer floor plan $F$, we know that it can be reconstructed from its eight labeled corner trees. The trees themselves can only be of exponentially many types, however the labeling is too expensive for our purposes. We have already seen that instead of vertex labels what we really need is bijections between the vertex sets of the trees. Even using Proposition~\ref{4trees} we would need three such bijections, and that is way more data than the $n^{n/3}$ bound we are aiming for. The solution is that instead of recording all the identifications between vertex sets, we will only record partial identifications, and use intrinsic geometric constraints of the floor plan to recover the missing identifications.

So let us consider the unlabeled (or, rather, only edge-labeled) versions of our corner trees $\Tti_1, \dots, \Tti_8$ associated to the two-level floor plan $F$.  The two layers will be in the $z$ direction. We will refer to the two layers as the top and the bottom. Corners $1,2,3,4$ are on the top, $5,6,7,8$ are on the bottom.

\begin{definition}[Identification set]
The full identification of a floor plan $F$ is the set
\[\FI = \{(x,y): x \in V(\Tti_i), y \in V(\Tti_j) \mbox{ where $x$ and $y$ represent the same block of $F$ in two different trees}\}.\]
An \textit{identification set} is any subset $\II \subset \FI$. An identification set $\II$ is \textit{strong} if, together with the eight unlabeled trees, it determines $\FI$.
\end{definition}

Our goal is then to find strong identification sets as small as possible. Theorem~\ref{randomconstr_theorem} shows that any strong identification set has to be of size at least $n/3$ asymptotically. We are going to show that for two-layer floor plans this can be achieved.  

This will complete the proof of Theorem~\ref{2layer_theorem}, by giving an injection from floorplans into the space of $8$-tuples of $\{x,y,z,d\}$-edge labeled quaternary trees and
strong identification sets of size $n/3$.  Note that a identification set consists of pairs $(x,y)$ so that no $x$ and no $y$ is in more than $8$ pairs.  Therefore the number
of identification set of size $t$ is given by a two (multi)sets of size $t$ and a bijection between them.  The number of such possibilities is at most 
\[
expo(n) \times (n/3)! = n^{(1+o(1))n/3}.
\]
Since the number of $\{x,y,z,d\}$-edge labeled quaternary trees is also exponential, we will have given an injection into a space of size $n^{(1+o(1))n/3}$, which will complete the proof.

First we need some notation. Blocks of a floor plan $F \in \TT_n$ can be broken into three classes, $T,B,D$ (as in top, bottom, double), denoting the blocks occupying only the top, only the bottom, or both layers.  For a block $x \in F$ let $x_i \in \Tti_i  (i = 1,\dots,8)$ denote the vertex in $\Tti_i$ which represents $x$.  It is in general impossible, given an $x_i \in \Tti_i$ to determine $x_j \in \Tti_j$ only knowing the unlabeled trees.  However certain elements of $\FI$  can be automatically reconstructed, as the following proposition shows.

\begin{proposition} \label{prop:free}
If $x \in T \cup D$, and $i,i' \in \{1,2,3,4\}$, $x_i$ can be determined from $x_i'$.  Likewise, if $x \in B \cup D$, and $i,i' \in \{5,6,7,8\}$ then $x_{i}$ can be determined from $x_i'$.
\end{proposition}
\begin{proof}
The top and bottom planes are 2D floor plans with blocks from $T \cup D$ and $B\cup D$ respectively.  The rooted subtree of $\Tti_i$, $1\leq i \leq 4$ obtained by deleting all $z$- and $d$-edges are precisely
the corner trees associated the corners of this 2D floor plan.  We know from~\cite{dual trees} that a 2D floor plan can be reconstructed from two of its corner trees corresponding to opposite corners. Furthermore, these trees do not have to be vertex-labeled. The bijection between their vertex sets can be recovered from the tree structure and edge-labeling. Hence looking at the subtrees obtained from $\Tti_1$ and $\Tti_3$ one can find the identification between the vertices in these two subtrees. Then the 2D floor plan can be reconstructed (up to topological equivalence). Finally, since the appropriate subtrees of  $\Tti_2$ and $\Tti_4$ are corner trees of the same 2D floor plan, their vertices can be identified with the blocks in this floor plan, and hence to the vertices of $\Tti_1$ and $\Tti_3$. The second part of the proposition follows identically.
\end{proof}

For the remainder, we prove the following:
\begin{proposition}\label{prop:smallstrong}
There exists a strong identification set $\II$ such that $|\II| \leq n/3$. 
\end{proposition}

This means that knowing the eight edge-labeled trees and $n/3$ well chosen identifications is enough to encode the floor plan. The cost of recording the trees is exponential, while the cost of writing down $n/3$ identifications is $n^{n/3} c^n$, hence this proves that $|\TT_n| \leq n^{n/3} c^n$ and thus completes the proof of Theorem~\ref{2layer_theorem}.

\begin{proof}[Proof of Proposition~\ref{prop:smallstrong}]

We begin by showing that identifying the double blocks allows one to reconstruct all identifications.  For any pair $i,j$ let us define
\[ \II_D^{i,j} = \{(x_i,x_j) : x \in D\} \mbox{ and } \II_D = \cup_{i,j} \II_D^{i,j}. \]

\begin{lemma} For any $i \in \{1,2,3,4\}$ and $j \in \{5,6,7,8\}$, the identification set  $\II_D^{i,j}$ is strong. \label{lem:complete}
\end{lemma}

\begin{remark} \label{rem:double}
Note that by Proposition \ref{prop:free} knowing $\II_D^{i,j}$ is equivalent to knowing $\II_D^{i',j'}$ for any for pair of $i',j'$ where $i' \in \{1,2,3,4\}$ and $j' \in \{5,6,7,8\}$, and also equivalent to knowing $\II_D$. Hence we can assume that the whole $\II_D$ is given.
\end{remark}

\begin{figure}[ht]
\centering
\begin{center}\def\svgwidth{2.5in}
$\begin{array}{cc}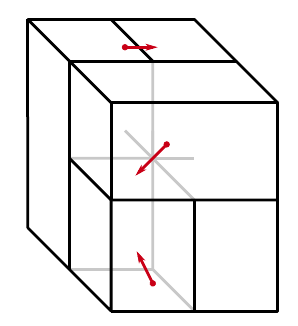&\def\svgwidth{2.5in}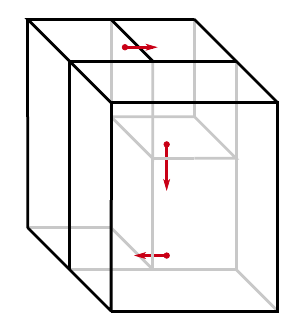
\end{array}$
\caption{Identification loops\label{fig:loop}}
\end{center}
\end{figure}

\begin{proof}[Proof of Lemma \ref{lem:complete}]
Our job is to find $\FI$ using $\II_D^{i,j}$. By Remark~\ref{rem:double} and by the symmetric role of the top and bottom layers, it is sufficient to show that we can find the part of $\FI$ corresponding to blocks in the top layer. By Proposition~\ref{prop:free} it is sufficient to find for each block $x \in T$ and each $j \in \{5,6,7,8\}$ the identification of $x_j$ with $x_i$ for some $i \in \{1,2,3,4\}$. Again, by symmetry, we can concentrate on $j = 5$, the other cases follow similarly.

Let's take a look at $\Tti_5$. Consider the rooted forest $\Tti_5'$ obtained by removing all $z$- and $d$-edges from it.  As we have seen in the proof of Proposition~\ref{prop:free}, the component of $\Tti_5'$ containing the original root corresponds to the blocks of $B \cup D$. All the other components consist of blocks from $T$. Now the crucial (but very easy) observations is that if a top block has an $x$- or $y$-type parent in $\Tti_5$, then it has to have the same type of parent in $\Tti_1$, and in fact the parent block has to be the same. Hence all these components of $\Tti_5'$ are isomorphically there in $\Tti_1$. If we can manage to identify the root vertex of each of these components of $\Tti_5'$ with the appropriate vertex in $\Tti_1$, then  the rest of the components are automatically identified by the local isomorphisms. Hence to finish the proof of the lemma, all we need is the following: given a vertex $u_5 \in \Tti_5$ whose parent is of type $z$ or $d$, find the corresponding vertex $u_1 \in \Tti_1$.

So let $u_5$ be a root vertex of a top component in $\Tti_5'$.  Then $u_5$ in $\Tti_5$ has either a $d$- or $z$-type parent. Let us denote this  parent by $v_5 \in \Tti_5$ corresponding to some block $v \in F$.  First assume it is a diagonal edge, as on the left side of Figure \ref{fig:loop}.  Since $v_5$ is in $B$, by Proposition~\ref{prop:free} we can find the corresponding $v_7 \in \Tti_7$. The parent block of this $y_7$ in $\Tti_7$ has to be necessarily a double block because of how the shape of a diagonal corner looks like. Let us denote this block by $w$. The edge of $y_7 w_7$ in $\Tti_7$ can only be an $x$-edge or an $y$-edge.  Without loss of generality we may assume the second case. (This is shown on the left side of Figure~\ref{fig:loop}.) Then the local geometry tells us that in $\Tti_2$ the parent of $w$ is $u$, the block we started from. Now since $w \in D$ and we know $w_7$, we can use the identification set $\II_D$ according to Remark~\ref{rem:double} to find its counterpart $w_2 \in \Tti_2$. The parent of $w_2$ has to be $u_2$ in $\Tti_2$. In summary, what happened was that we took a vertex $u_5 \in \Tti_5$, looked at a sequence of vertices in various trees, and finally arrived at a vertex of $\Tti_2$ that represents the same block. The sequence was
\[ u_5 \xrightarrow{d} v_5 \xrightarrow{B} v_7 \xrightarrow{y} w_7 \xrightarrow{\II_D} w_2 \xrightarrow{x} u_2 \]
where the first, third and last steps involved looking at the parent in the current tree, the second step used Proposition~\ref{prop:free} to find a bottom block in another tree, finally the fourth step used $\II_D$ to find a double block in another tree. (If the local configuration is the mirror image of what is seen in Figure~\ref{fig:loop}, then $w_5$ would be an $x$-type parent of $v_5$, then $u$ would be an $y$-type parent of $w$ in $\Tti_4$, hence by a similar argument we could find $u_4 \in \Tti_4$.)

Now assume that $v_5$ is a $z$-type parent of $u_5$,  as shown on the right side of Figure~\ref{fig:loop}. In this case we need to look at the parent of $v_5$ in $\Tti_5$ still. It is either an $x$- or a $y$-type parent. Again, without loss of generality, we can assume the second case, as shown in Figure~\ref{fig:loop}. Denoting by $w_5$ the parent of $v_5$, the local geometry says that $w$, the block represented by this vertex, has to be a double block. Further, in $\Tti_2$ the parent of $w$ has to be $u$. We have found $w_5 \in \Tti_5$. Since it is a double block, using $\II_D$ according to Remark~\ref{rem:double} we can find $w_2 \in \Tti_2$, and looking at its parent we find $u_2 \in \Tti_2$. The sequence now was
\[ u_5 \xrightarrow{z} v_5 \xrightarrow{y} w_5 \xrightarrow{\II_D} w_2 \xrightarrow{x} u_2. \]
(Again, in the mirrored case, if $w_5$ is an $x$-type parent of $v_5$, then $u_4$ would be the $y$-type parent of $w_4$ in $\Tti_4$, so we could find $u_4$ by the same argument.)

Thus in both cases, starting from a vertex in $\Tti_5$, we have found its counterpart in $\Tti_2$  (or in the mirrored cases in $\Tti_4$). But by Proposition~\ref{prop:free}, for any block in the top layer, identification between $\Tti_1$ and $\Tti_2$ (or $\Tti_4$) is free. So we can finally find $u_1 \in \Tti_1$. This completes the proof of the lemma.
\end{proof}

Next, we construct two more strong identification sets. To define these, we need a final bit of notation.
Consider a diagonal edge $u_5 v_5$ in $\Tti_5$. As we have seen in the proof of Lemma~\ref{lem:complete} and also in Figure~\ref{fig:loop}, the parent of the block $v$ in $\Tti_7$ is a double block $w$. Depending on the type of the $v_7 w_7$ edge, there are two cases. If the $v_7 w_7$ edge is of type $x$, then $u$ is the $y$-parent of $w$ in $Tti_4$. If the $v_7 w_7$ edge is of type $y$, then $u$ is the $x$-parent of $w$ in $\Tti_2$. Let us split the diagonal edges of $\Tti_5$ into two classes, $d_4$ and $d_2$, according to these two cases. (The subscript denotes the tree in which $u$ is the parent of $w$.)

We can similarly split the $z$-edges in $\Tti_5$ into two classes. In the $z_2$ class the child of the $z$-edge in $\Tti_5$ is the $x$-parent of a double block in $\Tti_2$ while in the $z_4$ class the child of the $z$-edge is the $y$-parent of a double block in $\Tti_4$. The various classes of edges in $\Tti_5$ are shown in Figure~\ref{fig:type}.

The same classification can be done for $\Tti_7$ instead. If $u_7 v_7$ is a diagonal (or $z$-) edge in $\Tti_7$ then the parent of $v$ in $\Tti_5$ (or in $\Tti_7$) is a double block $w$, and $u$ is either the $x$-parent of $w$ in $\Tti_4$ or the $y$-parent in $\Tti_2$. In the first case we call the edge of type $d_4$ (or $z_4$ respectively) and in the second case type $d_2$ (or $z_2$ respectively). The local arrangements corresponding to these situations would be mirror reflections of the ones in Figure~\ref{fig:type} through the vertical plane adjacent to corners 2,4,6,8.

\begin{figure}[ht]
\centering
\begin{center}\def\svgwidth{1.5in}
$\begin{array}{cccc}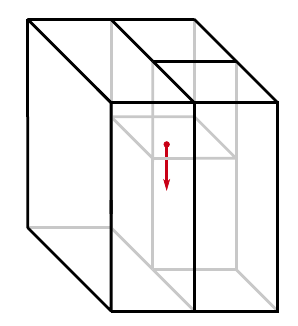&\def\svgwidth{1.5in}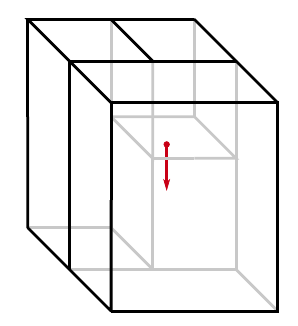&\def\svgwidth{1.5in}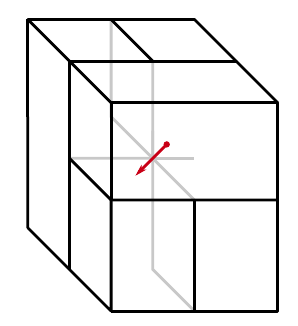&\def\svgwidth{1.5in}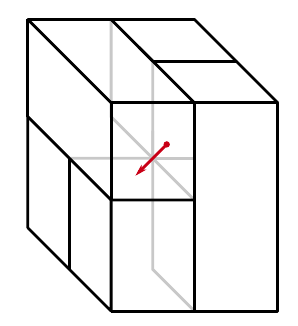\\
\mbox{$ZS$-type}&\mbox{$ZF$-type}&\mbox{$DS$-type}&\mbox{$DF$-type}
\end{array}$
\\
\caption{Refined edge types in $\Tti_5$\label{fig:type}}
\end{center}
\end{figure}

Now we are ready to define the two identification sets.
\[ \II_T^4 = \{ (u_5,u_4) : u_5 \in \Tti_5 \mbox{ is a $d_4$ or $z_4$ child} \} \cup \{ (u_7,u_4) : u_7 \in \Tti_7 \mbox{  is a $d_4$ or $z_4$ child} \} , \]
\[ \II_T^2 = \{ (u_5,u_2) : u_5 \in \Tti_5 \mbox{ is a $d_2$ or $z_2$ child} \} \cup \{ (u_7,u_2) : u_7 \in \Tti_7 \mbox{  is a $d_2$ or $z_2$ child} \} .\]

\begin{lemma}\label{lem:top}
The identification sets $\II_T^2$ and $\II_T^4$ are both strong.
\end{lemma}

\begin{proof}
Since the definitions of the two sets are symmetric, it suffices to show that $\II_T^2$ is strong. In the light of Lemma~\ref{lem:complete} if we can show that all the vertices representing double blocks in $\Tti_2$ can be identified with their counterparts in $\Tti_6$ then we are done.

We do this identification recursively: if a parent of a double block $w_2 \in \Tti_2$ is an other double block $u_2 \in \Tti_2$, then, since the whole floor plan has only two layers, the parent of $w_6 \in \Tti_6$ has to be the $u_6$, and the two parents have to have the same type. Hence if we could identify $u_2$ with $u_6$, and we find that the parent of $w_2$ is $w_6$ of type $x$, then by looking at the $x$-child of $u_6$ we can find $u_2$. (Similarly for $y$-type parents.) Thus we only have to worry about double blocks $w$ for which the parent of $w_2 \in \Tti_2$ is a single block (hence of type $T$) $u_2 \in \Tti_2$.

Suppose first that $u_2$ is an $x$-parent of $w_2$. Then the local arrangement has to look like as in Figure~\ref{fig:type}. Thus the parent $v$ of the top block $u$ in $\Tti_5$ has to be a $z$- or $d$-parent, furthermore it can only be of type $z_2$ or $d_2$. This means that $(u_5,u_2) \in \II_T^2$. We can use this to trace an identification loop as before:
\[ w_2 \xrightarrow{x} u_2 \xrightarrow{\II_T^2} u_5 \xrightarrow{d_2} v_5 \xrightarrow{B} v_7 \xrightarrow{y} w_7 \xrightarrow{B} w_6, \] or
\[ w_2 \xrightarrow{x} u_2 \xrightarrow{\II_T^2} u_5 \xrightarrow{z_2} v_5  \xrightarrow{x} w_5 \xrightarrow{B} w_6, \]
depending on type of the $u_5 v_5$ edge. In other words, starting from $w_2$, and looking at parents or jumping between trees using the known identifications, we find its counterpart $w_6 \in \Tti_6$.

Similarly, if $u_2$ is an $y$-parent of $w_2$. Then the local arrangement has to look like the mirror image of those in Figure~\ref{fig:type}. Thus the parent $v$ of the top block $u$ in $\Tti_7$ has to be a $z$- or $d$-parent, furthermore it can only be of type $z_2$ or $d_2$. The argument is then the same as above. This completes the proof of the lemma.
\end{proof}

Now it is very easy to finish proving the proposition. Clearly, $|\II_D^{i,j}| = |D|$. On the other hand, observe, that $\II_T^2 \cup \II_T^4$ has exactly one identification written down for each vertex having a $z$- or a $d$-parent in $\Tti_5$, and one identification for each vertex having a $z$- or $d$-parent in $\Tti_7$. But only top blocks can have $z$- or $d$-parents, hence $ | \II_T^2 \cup \II_T^4 | \leq 2|T|$, and so one of them has to be at most of size $|T|$. Finally one could repeat the whole argument of Lemma~\ref{lem:top} with identifications sets $\II_B^6$ and $\II_B^8$ defined entirely analogously, reversing top and bottom everywhere. Then one of these two strong identification sets have to be of size at most $|B|$.

So we have found three identification sets of size at most $|D|, |T|$, and $|B|$ respectively. Since $|D| + |T| + |B| = n$, one of these have to be of size at most $n/3$, proving the proposition, completing the proof of Theorem \ref{2layer_theorem}.
\end{proof}


\begin{thebibliography}{99}
\bibitem{thesis?} E. Ackerman. Counting problems for geometric structures: rectangulations, floor plans, and quasi-planar graphs. PhD Thesis, 2006, Technion, Haifa
\bibitem{baxter}  E.l Ackerman, G. Barequet, and R. Y. Pinter. A bijection between permutations and floorplans, and its applications. {\it Discrete Applied Mathematics (DAM)}, \textbf{154}:12 (2006), 1674--1684.
\bibitem{flk} R. Fischbach, J. Leinig, J. Knechtel, Investigating modern layout representations for improved 3d design automation, Proceedings of {\it 21st Great lakes symposium on VLSI}, pp. 337--342, 2011
\bibitem{3dcorn} Y. Ma, X. Hong, S. Dong, C-K. Cheng, 3D CBL: an efficient algorithm for general 3D packing problems., Proceedings of {\it 48th Midwest Symposium on Circuits and Systems}, pp. 1079--1082, 2005
\bibitem{OSeq} H. Ohta, T. Yamada, C. Kodama, and K. Fujiyosi, The O-Sequence: Representation of 3D-floorplan dissected by rectangular walls, {\it Prime '06}, pp. 317-320, 2006.
rectangular walls
\bibitem{dual trees} B. Yao, H. Chen, C-K. Cheng, and R. Graham. Floorplan representations: complexity and connections. {\it ACM Trans. on Design Automation}
of Electronic Systems, pp. 55--80, Jan. 2003
\bibitem{Cheng1} R. Wang, E.F.Y. Young, and C-K. Cheng. Representing topological structures for 3-D floorplanning.{\it  ICCCAS}, pp.1098--1102, July 2009
\bibitem{3dC} P. Garrou, C. Bower, and P. Ramm (eds), {\it Handbook of 3D Integration: Technology and Applications of 3D Integrated Circuits, Vols. 1, 2}, Wiley-VCH, 2008.
\end{thebibliography}
\end{document}